\documentclass[11pt]{amsart}
 \usepackage[utf8]{inputenc}
\usepackage[english]{babel}
\usepackage[T1]{fontenc}
\usepackage{mathtools}
\usepackage{pdfpages}
\usepackage{changes}
\usepackage{tikz-cd}
\usepackage{soul}
\usepackage{float}
\usepackage{multicol}
\usepackage{amsthm,amssymb,url,hyperref}
\usepackage{fullpage}
\usepackage{stmaryrd}
\usepackage[all]{xy}
\usepackage{color}

\theoremstyle{plain}
\newtheorem{theorem}{Theorem}[section]

\newtheorem{corollary}[theorem]{Corollary}

\newtheorem{lemma}[theorem]{Lemma}

\newtheorem{proposition}[theorem]{Proposition}

\theoremstyle{definition}

\newtheorem*{problem*}{Problem}
\newtheorem*{theorem*}{Theorem}

\newtheorem{remark}[theorem]{Remark}

\def\ker#1{\mathrm{ker}(#1)}

\def\aut#1{\mathrm{Aut}(#1)}

\def\setof#1#2{\{#1\, : \,#2\}}

\newcommand{\Aut}{\operatorname{Aut}}

\newcommand{\Hol}{\operatorname{Hol}}
\def\cg#1{\equiv_\alpha}
\newcommand*\xbar[1]{%
   \hbox{%
     \vbox{%
       \hrule height 0.5pt 
       \kern0.5ex
       \hbox{%
         \kern-0.1em
         \ensuremath{#1}%
         \kern-0.1em
       }%
     }%
   }%
} 


\title{Skew Braces of size $pq$}

\begin{document}



%

\begin{abstract}
    We construct all skew braces of size $pq$ (where $p>q$ are primes) by using Byott's classification of Hopf--Galois extensions of the same degree. For $p\not\equiv 1 \pmod{q}$ there exists only one skew brace which is the trivial one. When $p\equiv 1 \pmod{q}$, we have $2q+2$ skew braces, two of which are of cyclic type (so, contained in Rump's classification) and $2q$ of non-abelian type.
\end{abstract}

\author{E. Acri}
\author{M. Bonatto}

\address[E. Acri, M. Bonatto]{IMAS--CONICET and Universidad de Buenos Aires, 
Pabell\'on~1, Ciudad Universitaria, 1428, Buenos Aires, Argentina}
\email{eacri@dm.uba.ar}

\email{marco.bonatto.87@gmail.com}

\maketitle

\section*{Introduction}

In the last few decades, there was an increasing interest in studying solutions to the set-theoretical Yang--Baxter equation (YBE). Following Drinfeld (\cite{MR1183474}), who first stated the problem, we say that for a given set $X$ and a function $r:X\times X\to X\times X$, the pair $(X,r)$ is a \emph{set-theoretical solution to the Yang--Baxter equation} if
\begin{equation}\label{YBE}
(id_X \times r)(r\times id_X)(id_X \times r)=(r\times id_X)(id_X \times r)(r\times id_X)
\end{equation}
holds.
A particular family of solutions is the family of \emph{non-degenerate} solutions, i.e. solutions obtained as
\begin{equation}\label{non-deg}
r:    X\times X\longrightarrow X\times X,\quad (x,y)\mapsto (\sigma_x(y),\tau_y(x))
\end{equation}
where $\sigma_x,\tau_x$ are permutations of $X$ for every $x\in X$. Non-degenerate solutions have been studied by several different authors \cite{MR1722951,MR1637256, MR1769723,MR1809284}.

Examples of non-degenerate involutive solutions to YBE are provided by \emph{braces}, introduced by Rump as a generalization of radical rings \cite{MR2278047}. In \cite{MR3177933} Cedó, Jespers and Okni\'nski settled an equivalent definition that was generalized later to \emph{skew (left) braces} by Guarnieri and Vendramin in \cite{MR3647970}. Skew braces
allow us to study non-involutive non-degenerate solutions.

A \emph{skew (left) brace} is a triple $(A,+,\circ)$ where $(A,+)$ and $(A,\circ)$ are groups (not necessarily abelian) such that
\[
a\circ(b+c)=a\circ b-a+a\circ c
\]
holds for every $a,b,c\in A$. Braces are skew braces for which the additive group is abelian and a skew brace $(A,+,\circ)$ is said to be a \emph{bi-skew} brace if also $(A,\circ,+)$ is a skew brace \cite{biskew}. 

The problem of finding non-degenerate solutions of \eqref{YBE} can be reduced to the classification problem of skew braces. Indeed, given an involutive non-degenerate solution as in \eqref{non-deg}, the group generated by $\setof{\sigma_x}{x\in X}$ has a canonical structure of brace. Bachiller, Cedó and Jespers \cite{MR3527540} provided a construction for all involutive non-degenerate solutions to the Yang--Baxter equation with a given brace structure over such group. In \cite{MR3835326}, Bachiller generalizes that construction by considering a \emph{permutation group} related to a non-degenerate solution which in turn have a natural structure of skew brace, \cite[Theorem 3.11]{MR3835326}. Therefore, in that work the classification of all non-degenerate solutions is reduced to the classification of all skew braces.

There was considerably recent progress on the classification problem for (skew) braces. Braces with cyclic additive group and braces of size $p^2q$ for $p,q$ primes with $q>p+1$ have been classified respectively in \cite{MR2298848, Rump} and \cite{Dietzel}. Bachiller \cite{MR3320237} solved the problem for braces of order $p^3$ where $p$ is a prime and Nejabati Zenouz \cite{NZ} completed the classification of skew braces of order $p^3$. In \cite{NZpaper}, Nejabati Zenouz include the automorphism groups of the skew braces of size $p^3$ of Heisenberg type. In this paper we provide a solution to the following problem. 
\begin{problem*} \cite[Problem 2.15]{problems} Let $p$ and $q$ be different prime numbers. Construct all skew left braces of size $pq$ up to isomorphism.
\end{problem*}

Our classification is based on the algorithm for the construction of skew braces with a given additive group developed in \cite{MR3647970}. Indeed, the algorithm allows to obtain all the skew braces with additive group $A$ from regular subgroups of its holomorph $\Hol(A)=A\rtimes\Aut(A)$. The isomorphism classes of skew braces are parametrized by the orbits of such subgroups under the action by conjugation of the automorphism group of $A$ in $\Hol(A)$ \cite[Section 4]{MR3647970}. 


The connection between skew braces, regular subgroups and Hopf--Galois extensions was observed by Bachiller \cite[Remark 2.8]{MR3465351}. We refer the reader to the Appendix of \cite{MR3763907} for further details on such connection.

A purely group theoretical approach to Hopf--Galois extension was pointed out by Byott in \cite{Byott_1}, in order to enumerate Hopf--Galois extensions over a given group. In particular Hopf--Galois extensions of degree $pq$ where $p>q$ are primes have been investigated in \cite{Byott_pq} through an explicit description of regular subgroups.

We are using such description as the first step of the classification for skew braces of size $pq$. Our main results are summarized in the following theorem.


\begin{theorem*}
	Let $p>q$ be primes. If $p\not\equiv 1\pmod q$ there is only one skew brace of size $pq$, the trivial one. If $p\equiv 1\pmod q$, a complete list of the $2q+2$ skew braces of order $pq$, up to isomorphism, is the following, where $g$ is a fixed element of $\mathbb{Z}_p$ of multiplicative order $q$: 	
	\begin{itemize}
 	    \item Additive group $\mathbb{Z}_{p}\times \mathbb{Z}_q$:
 	    \begin{eqnarray*}
\begin{pmatrix} n \\ m \end{pmatrix} + \begin{pmatrix} s \\ t \end{pmatrix} =\begin{pmatrix} n+s \\ m+t \end{pmatrix};
	\end{eqnarray*}
	        \begin{enumerate}
	            \item[(i)] the trivial skew brace over $\mathbb{Z}_{p}\times \mathbb{Z}_q$;
	            \item[(ii)] the bi-skew brace $(A,+,\circ)$ where
	            \begin{eqnarray*}
\begin{pmatrix} n \\ m \end{pmatrix} \circ \begin{pmatrix} s \\ t \end{pmatrix} =\begin{pmatrix} n+g^{m}s \\ m+t \end{pmatrix}.
		\end{eqnarray*}
	        \end{enumerate}
	        
	   \item Additive group $\mathbb{Z}_p\rtimes_g \mathbb{Z}_q$:
	   \begin{eqnarray*}
\begin{pmatrix} n \\ m \end{pmatrix} + \begin{pmatrix} s \\ t \end{pmatrix} =\begin{pmatrix}  n+g^{m}s \\ m+t \end{pmatrix};
	\end{eqnarray*}
	        \begin{enumerate}
	            \item[(i)] the trivial skew brace over $\mathbb{Z}_p\rtimes_g \mathbb{Z}_q$;
	            \item[(ii)] the skew brace $(A,+,\circ)$ where	\begin{eqnarray*}
\begin{pmatrix} n \\ m \end{pmatrix} \circ \begin{pmatrix} s \\ t \end{pmatrix} =\begin{pmatrix} g^t n+g^{m}s \\ m+t \end{pmatrix};
	\end{eqnarray*}
	            \item[(iii)] the bi-skew braces $A_\gamma=(A,+,\circ)$ for $1<\gamma\leq q$ where 
\begin{eqnarray*}
\begin{pmatrix} n \\ m \end{pmatrix} \circ \begin{pmatrix} s \\ t \end{pmatrix} =\begin{pmatrix} n+\left(g^{\gamma}\right)^m s \\ m+t \end{pmatrix};   
\end{eqnarray*}
	            \item[(iv)] the skew braces $A_\mu=(A,+,\circ)$ for $1< \mu\leq q$ where%
\begin{eqnarray*}
\begin{pmatrix} n \\ m \end{pmatrix} \circ \begin{pmatrix} s \\ t \end{pmatrix} =\begin{pmatrix} g^t n+\left(g^{\mu}\right)^m s \\ m+t \end{pmatrix}.
\end{eqnarray*}
		 
\end{enumerate}
	\end{itemize}
\end{theorem*}
Our result agrees with the enumeration formula for skew braces of square free size recently presented by Byott and Alabdali, see \cite[Section 7.2]{byott_squarefree}.

This paper is organized as follows: in Section \ref{preliminaries} we collect some basic definitions and we present more details about the classification strategy. In Section \ref{groups_pq} we describe the groups of size $pq$ and their automorphisms. In Section \ref{classification} we deal with the classification computing the orbits of regular subgroups in the holomorph of the relevant groups leading to the main result. 

\section{Preliminaries}\label{preliminaries}

A triple $(A,+,\circ)$ is said to be a \emph{skew (left) brace} if both $(A,+)$ and $(A,\circ)$ are groups and $$a\circ(b+c)=a\circ b-a+a\circ c$$ holds for every $a,b,c\in A$. Following the standard terminology for Hopf--Galois extensions, if $\chi$ is a group theoretical property, we say that a skew brace $(A,+,\circ)$ is of \emph{$\chi$-type} if $(A,+)$ has the property $\chi$. 

Given a skew brace $(A,+,\circ)$, the group $(A,\circ)$ acts on $(A,+)$ by automorphisms. Indeed the mapping
\begin{equation}\label{axiom}
\lambda:(A,\circ)\to \Aut(A,+), \quad \lambda_a(b)=-a+a\circ b,
\end{equation}
is a homomorphism of groups.

A {\it bi-skew brace} is a skew brace $(A,+,\circ)$ such that $(A,\circ,+)$ is also a skew brace (see \cite{biskew}). Equivalently
\begin{equation}\label{eq for biskew}
x+(y\circ z)=(x+ y)\circ x'\circ (x+ z) 
\end{equation}
holds for every $x,y,z\in A$, where $x'$ denotes the inverse of $x$ in $(A,\circ)$. We provide a construction of bi-skew braces over semidirect products of groups. Note that the first part is a special case of \cite[Proposition 4.6.12]{NZ}. We include the proof for completeness.

\begin{proposition}\label{biskew}
	Let $A$ be a group, $B$ be an abelian group and $\eta,\rho:B\longrightarrow \aut{A}$ be group homomorphisms. If $[Im(\rho),Im(\eta)]=1$ then $G_{\eta,\rho}=(G,+,\circ)$ where $(G,+)=A\rtimes_{\eta} B$ and $(G,\circ)=A\rtimes_{\rho} B$ is a bi-skew brace with $A\times \{0\}\leq \ker{\lambda}$.
\end{proposition}
\begin{proof}
		Let us denote by $\rho_x$ the image of $\rho$ and by $\eta_x$ the image of $\eta$ for every $x\in $B. Then
			\begin{eqnarray*}
			(x,s)\circ\left((y,t)+(z,u)\right)&=&(x,s)\circ(y\eta_t(z),t+u)\\
			&=&(x\rho_s(y)\rho_s(\eta_t(z)),s+t+u)
		\end{eqnarray*}	
	and		
			\begin{eqnarray*}
			\left((x,s)\circ(y,t)\right)- (x,s)+ \left((x,s)\circ(z,u)\right)&=&
			(x\rho_s(y),s+t)+(\eta_{-s}(x)^{-1},-s)+(x\rho_s(z),s+u)\\
			&=&(x\rho_s(y)\eta_t(x)^{-1},t)+(x\rho_s(z),s+u)\\
			&=&(x\rho_s(y)\eta_t(x)^{-1}\eta_t(x)\eta_t(\rho_s(z)),t+s+u)\\
			&=&(x\rho_s(y)\eta_t(\rho_s(z)),t+s+u)
		\end{eqnarray*}
for every $x,y,z\in A$ and $s,t,u\in B$. Since $\eta_t$ and $\rho_s$ commute then $(A,+,\circ)$ is a skew brace. The same argument shows that $(A,\circ,+)$ is a skew brace as well. Moreover,
\begin{eqnarray}\label{formula for lambda}
\lambda_{(x,s)}(y,t)&=&-(x,s)+(x,s)\circ (y,t)=(\eta_{-s}(x)^{-1},-s)+(x\rho_s(y),t+s)=\notag\\
&=&(\eta_{-s}(x)^{-1}\eta_{-s}(x)\eta_{-s}(\rho_s(y)),t)=\notag\\
&=&(\eta_{-s}(\rho_s(y)),t).
\end{eqnarray}
for every $x,y\in A$ and $s,t\in B$. Then $A\times \{0\}\leq \ker{\lambda}$.
\end{proof}
Proposition \ref{biskew} applies to cyclic groups, since they have abelian automorphism groups.
\begin{corollary}\label{biskew on cyclic}
Let $A$ be a cyclic group and $B$ be an abelian group, $(G,+)=A\rtimes_{\eta} B$ and $(G,\circ)=A\rtimes_{\rho} B$. Then $G_{\eta,\rho}$ is a bi-skew brace and $A\times \{0\}\leq \ker{\lambda}$.
\end{corollary}

The \emph{holomorph} of a group $(A,+)$ is the semidirect product $A\rtimes \Aut(A)$ where the operation is
\[
(a,f)(b,g)=(a+f(b),fg),
\]
for every $a,b\in A$ and $f,g\in \aut{A}$. We denote by $\pi_1:\Hol(A)\to A$ and $\pi_2:\Hol(A)\to \Aut(A)$ the canonical surjections. Note that $\Hol(A)$ acts on $A$ by $(a,f)(b)=a+f(b)$, for all $(a,f)\in \Hol(A)$ and $b\in A$. Thus $\Hol(A)$ is a group of permutations on $A$.

Recall that a group of permutations $N$ on a set $X$ is said to be \emph{regular} if, given any $x,y\in X$, there is a unique $g\in N$ such that $g(x)=y$. 

From \cite{MR3647970}, we know that given a group $(A,+)$ (not necessarily abelian) we have a bijective correspondence between isomorphism classes of skew braces $(A,+,\circ)$ and orbits of regular subgroups of $\Hol(A,+)$ under conjugation by $\Aut(A,+)$ identified with the subgroup $\{1\}\times \aut{A}\leq \Hol(A,+)$. If $G$ is a regular subgroup of $\Hol(A,+)$, then it is easy to verify that the map $\pi_1|_G:G\to A$ is a bijection.

\begin{theorem}\cite[Theorem 4.2, Proposition 4.3]{MR3647970}\label{thm:skew_holomorph}
Let $(A,+)$ be a group. If $\circ$ is an operation such that $(A,+,\circ)$ is a skew brace, then $\{ (a,\lambda_a): a\in A \}$ is a regular subgroup of $\Hol(A,+)$. Conversely, if $G$ is a regular subgroup of $\Hol(A,+)$, then $A$ is a skew brace with $$a\circ b=a+f(b)$$ where $(\pi_1|_G)^{-1}(a)=(a,f)\in G$ and $(A,\circ)\cong G$.

Moreover, isomorphism classes of skew braces over $A$ are in bijective correspondence with the orbits of regular subgroups of $\Hol(A)$ under the action of $\aut{A}$ by conjugation.
\end{theorem}


An explicit description of all the regular subgroups of $\Hol(A)$ for groups of size $pq$, where $p,q$ are different primes, is given in \cite{Byott_pq}. So, we can construct all skew braces with additive group isomorphic to $A$ by finding representatives of the orbits of regular subgroups under the action by conjugation of $\Aut(A)$ on $\Hol(A)$ and provide an explicit formula for the operations of such skew braces using Theorem \ref{thm:skew_holomorph}.  Following \cite{Byott_pq}, we will denote by $e^\prime(G,A,m)$ the number of the regular subgroups of $\Hol(A)$ isomorphic to a given group $G$ such that their image under $\pi_2$ has size $m$. 

\begin{remark}\label{remark for lambdas}
Let $(A,+)$ be a group and $G$ a regular subgroup of $\Hol(A)$. According to Theorem \ref{thm:skew_holomorph}, $(A,+,\circ)$ where 
\[
a\circ b=a+\pi_2((\pi_1|_G)^{-1}(a))(b)
\]
for every $a,b\in A$ is a skew brace. In other words, $\lambda_a=\pi_2(\pi_1|_G^{-1}(a))$ and 
 so $|\ker{\lambda}|=\frac{|G|}{|\pi_2(G)|}$.
\end{remark}

\section{Groups of order {\it pq} and their automorphism groups}\label{groups_pq}

In this section we describe the groups of order $pq$ up to isomorphism and their automorphisms, where $p>q$ are primes.

 If $p\not\equiv 1 \pmod {q}$, the unique group of size $pq$ is the cyclic one, for which we are using the following presentation: 
\begin{equation}\label{eq:presentation_C}
C=\langle \sigma, \tau\, | \,\sigma^p=\tau^q=1,\, \tau\sigma=\sigma\tau \rangle.
\end{equation}

The group $\Aut(C)$ is isomorphic to $\mathbb{Z}_{p-1}\times \mathbb{Z}_{q-1}$ and it has the following presentation:
\begin{equation}\label{aut of M}
    \Aut(C)=\langle \phi, \psi \mid \phi^{p-1}=\psi^{q-1}=1,\, \phi\psi=\psi\phi \rangle
\end{equation}
where the morphisms $\phi$ and $\psi$ are defined by setting 
\begin{align*}
\phi(\sigma)&=\sigma^n,  & \psi(\sigma)&=\sigma,\\
\phi(\tau)&=\tau,  &\psi(\tau)&=\tau^{m},
\end{align*}
provided that $n$ has multiplicative order $p-1$ modulo $p$ and $m$ has multiplicative order $q-1$ modulo $q$.

 If $p\equiv 1 \pmod {q}$, there exists also a unique non-abelian group, namely
\begin{equation}\label{eq:presentation_M}
M=\langle \sigma, \tau\, |\, \sigma^p=\tau^q=1,\, \tau\sigma=\sigma^g\tau \rangle\cong \mathbb{Z}_p\rtimes_g \mathbb{Z}_q
\end{equation}
where $g$ has multiplicative order $q$ modulo $p$. Following the same notation as in \cite{Byott_pq}, we denote by $a_0$ an integer satisfying 
\begin{equation}\label{eq:a_0}
    (g-1)a_0\equiv 1 \pmod {p}.
\end{equation}
We are using the following well-known formula with no further reference:
\begin{equation}\label{eq:geometric}
    \sum_{i=0}^{r-1} g^i \equiv \frac{g^r-1}{g-1}\equiv a_0(g^r-1) \pmod{p}
\end{equation}
for every $r\in \mathbb{N}$.

A description of the elements of the group $\Aut(M)$ is given by the following lemma.
\begin{lemma}\cite[Lemma 2.3]{Ghorbani2014}\label{autom}
 The automorphisms of $M$ are $$\setof{\varphi_{i,j}}{1\leq i\leq p-1,\, 0\leq j\leq p-1}$$ 
 defined by setting $\varphi_{i,j}(\sigma)=\sigma^i$ and $\varphi_{i,j}(\tau)=\sigma^j\tau$.  
\end{lemma}

\section{Skew Braces of size {\it pq}}\label{classification}

In this section, $p>q$ are primes.

\subsection{Trivial skew braces}

A skew brace $(A,+,\circ)$ is said to be {\it trivial} if $a+b=a\circ b$ for every $a,b\in A$. For every group $G$ there exists a unique trivial skew brace $A$ with $G=(A,+)=(A,\circ)$. Therefore if $p\equiv 1\pmod {q}$, there exists two trivial skew braces of size $pq$. On the other hand, if $p\not\equiv 1\pmod q$ there exists a unique skew brace of size $pq$ and it is trivial.
\begin{proposition}
Let $p>q$ be primes such that $p\not\equiv 1 \pmod q$. There exists a unique skew brace of size $pq$ and it is trivial with $(A,+)= (A,\circ)\cong \mathbb{Z}_{pq}$.
\end{proposition}

\begin{proof}
If $p\not\equiv 1 \pmod {q}$, then $pq$ is a Burnside number, so there exists a unique skew brace of order $pq$ (\cite[Theorem A.8]{MR3763907}).
\end{proof}

\subsection{Skew braces of cyclic-type}\label{abelian type}

In this section we classify skew braces of cyclic type of size $pq$. This case is covered by the classification given in \cite{MR2298848,Rump} and we include it here for completeness. From now on, we consider the case $p\equiv 1 \pmod{q}$.

    Any non-trivial homomorphic image of a group of order $pq$ in $\aut{C}$ has size $q$, since $|\aut{C}|=(p-1)(q-1)$ and $q$ divides $p-1$. 
    So, we only have to consider regular subgroups such that the image under $\pi_2$ has size $q$.
In the following we will denote by $\alpha$ the automorphism of $C$ defined by 
$$\alpha(\sigma)=\sigma^g,\quad \alpha(\tau)=\tau.$$
\begin{lemma}\label{pro:cyclic_type}\cite[Lemma  4.1]{Byott_pq}
The regular subgroups of $\Hol(C)$ such that the image under $\pi_2$ has size $q$ are 
\begin{equation}
G_b=\langle \sigma, \tau^b\alpha \rangle\cong M,
\end{equation}		
for $1\leq b\leq q-1$. In particular, 
$e^\prime(C,C,q)=0$ and $e^\prime(M,C,q)=q-1$.

	\end{lemma}

The enumeration of skew braces with cyclic additive group follows from the enumeration of the orbits of regular subgroups of $\Hol(C)$.
\begin{proposition}\label{enum_ab_type}
There exists a unique non-trivial skew brace of abelian type of size $pq$. 
\end{proposition}

\begin{proof}
It is enough to show that all the subgroups of $\Hol(C)$ given in Lemma \ref{pro:cyclic_type} are conjugate by some element in $\aut{C}$. Indeed, 
	\begin{align*}
	\psi^jG_b\psi^{-j} &=\langle \psi^j\sigma\psi^{-j}, \psi^{j}\tau^b\alpha\psi^{-j} \rangle \\
	&= \langle \psi^{j}(\sigma), \psi^{j}(\tau)^b\alpha \rangle 
	\\
	&= \langle \sigma, \tau^{m^j b}\alpha \rangle \\
	&= G_{m^j b}.
	\end{align*}
	Since $m$ has multiplicative order $q-1$ modulo $q$, all the subgroups $G_b$ are conjugate.
\end{proof}

\begin{theorem}
	The unique non-trivial skew brace of cyclic type
	is $(A,+,\circ)$ where $(A,\circ)\cong\mathbb{Z}_p\rtimes_g\mathbb{Z}_q$, i.e. 
		\begin{eqnarray*}
 \begin{pmatrix} n \\ m \end{pmatrix} + \begin{pmatrix} s \\ t \end{pmatrix} =\begin{pmatrix} n+s \\ m+t \end{pmatrix},\qquad 
\begin{pmatrix} n \\ m \end{pmatrix} \circ \begin{pmatrix} s \\ t \end{pmatrix} =\begin{pmatrix} n+g^{m}s \\ m+t \end{pmatrix}
		\end{eqnarray*}
	for every $0\leq n,s\leq p-1,\, 0\leq m,t\leq q-1$. In particular, $(A,+,\circ)$ is a bi-skew brace.
\end{theorem}
\begin{proof}
According to Corollary \ref{biskew on cyclic}, $(A,+,\circ)$ is a bi-skew brace with $|\ker{\lambda}|=p$. By virtue of Proposition \ref{enum_ab_type}, $(A,+,\circ)$ is the unique skew brace with such properties.
\end{proof}

\subsection{Skew braces of non-abelian type}\label{non-abelian type}



In this section we also assume that $p$ and $q$ are primes such that $p\equiv 1\pmod q$. In order to enumerate non-trivial skew braces with additive group isomorphic to $M$ we need to compute $e^\prime(G,M,m)$ for $m\in\{p,q,pq\}$ and $G=C,M$.

The only subgroup of order $pq$ in $\Aut(M)$ is the subgroup generated by 
\begin{equation}\label{def of alpha and beta}
\alpha=\varphi_{1,1},\quad \beta=\varphi_{g,0}    
\end{equation}
 as defined in Lemma \ref{autom}.
Hence, for every skew brace of order $pq$, the image of $\lambda$ lie in such subgroup.

\begin{lemma}\label{lem:ker=q}\cite[Lemma 5.1]{Byott_pq}
The regular subgroups of $\Hol(M)$ such that the image under $\pi_2$ has size $p$ are 
\begin{equation}\label{reg groups ker=p}
	G_c=\langle \sigma^{a_0}\alpha, \sigma^c\tau \rangle\cong C
	\end{equation}
where $0\leq c\leq p-1$. In particular, 
	$e^\prime(M,M,p)=0$ and
	$e^\prime(C,M,p)=p$. 
\end{lemma}


	

	


\begin{theorem}
	 There exists a unique skew brace of size $pq$ with additive group $M$ and $|\ker{\lambda}|=q$. Such skew brace is $(A,+,\circ)$ where
	\begin{eqnarray}
    \begin{pmatrix} n \\ m \end{pmatrix} + \begin{pmatrix} s \\ t \end{pmatrix} =\begin{pmatrix} n+g^{m}s \\ m+t \end{pmatrix},\qquad 
\begin{pmatrix} n \\ m \end{pmatrix} \circ \begin{pmatrix} s \\ t \end{pmatrix} =\begin{pmatrix} g^t n+g^{m}s \\ m+t \label{formula for q-case}\end{pmatrix}
	\end{eqnarray}
		for every $0\leq n,s\leq p-1,\, 0\leq m,t\leq q-1$. \end{theorem}
\begin{proof}
	We first show that all the groups $G_c$ as in \eqref{reg groups ker=p} are conjugate to $G_0$ by an element of $\aut{M}$. 
Let $0\leq c\leq p-1$ and set $\eta=\varphi_{1,-c}$, so $\eta^{-1}=\varphi_{1,c}$. Indeed, since $\eta$ centralizes $\sigma^{a_0}\alpha$, we have that
		\begin{eqnarray*}
		\eta G_c\eta^{-1} 
			= \langle \eta(\sigma^{a_0})\eta\alpha\eta^{-1},\eta(\sigma^c\tau) \rangle 		= \langle \sigma^{a_0}\alpha,\tau \rangle=G_0.
		\end{eqnarray*}
So there exists a unique skew brace with the desired properties.
It is straightforward to verify that $(A,+,\circ)$ as defined in $\eqref{formula for q-case}$ is a skew brace and $|\ker{\lambda}|=q$. 
\end{proof}

\begin{lemma}\label{lem:ker=p}\cite[Lemma 5.2]{Byott_pq} The regular subgroups of $\Hol(M)$ such that the image under $\pi_2$ has size $q$ are 
\begin{equation}\label{reg groups ker=q}
G_{a,b}=\langle \sigma, \tau^a\alpha^b\beta \rangle,
	\end{equation}
		where $1\leq a\leq q-1$ and $0\leq b\leq p-1$. In particular, $e^\prime(M,M,q)=p(q-2)$ and $e^\prime(C,M,q)=p$.
\end{lemma}
	

\begin{proposition}\label{enum_ker=p}
	There are $q-1$ skew braces with additive group $M$ and $|\ker{\lambda}|=p$. 
\end{proposition}	
	
	\begin{proof}
		We show that $G_{a,0}$ for $1\leq a\leq q-1$ as defined in \eqref{reg groups ker=q} form a set of representatives of the orbits.	If $b\neq 0$,
	\begin{equation*}
		\varphi_{1,-ba_0}G_{a,0}\varphi_{1,-ba_0}^{-1}=	
		\langle \sigma, (\sigma^{-ba_0}\tau)^a\alpha^b\beta \rangle = \langle \sigma, \tau^a\alpha^b\beta \rangle = G_{a,b},
	\end{equation*}
where we multiplied the second generator by a suitable power of $\sigma$ to obtain the second equality. Also, we used the identities	$\varphi_{i,j}\alpha\varphi_{i,j}^{-1}=\alpha^i$ and 	$\varphi_{i,j}\beta\varphi_{i,j}^{-1}=\varphi_{g,(1-g)j}=\alpha^{(1-g)j}\beta$.
			
	
	
The subgroup generated by $\sigma$ is invariant under the action by conjugation of $\aut{M}$. So, $G_{a,0}$ and $G_{b,0}$ are conjugate if and only if 
\begin{equation}\label{element to check}
\varphi_{i,j} \tau^a \beta \varphi_{i,j}^{-1}=\varphi_{i,j}(\tau)^a\varphi_{g,(1-g)j}=\sigma^{j\frac{g^a-1}{g-1}}\tau^a\varphi_{g,(1-g)j}\in G_{b,0}\end{equation}
for some $\varphi_{i,j}\in\aut{M}$. Then $\tau^a\varphi_{g,(1-g)j}\in G_{b,0}$ since $\sigma\in G_{b,0}$ and $\pi_2(\tau^a\varphi_{g,(1-g)j})=\varphi_{g,(1-g)j}=\beta^k$ for some $k\in \mathbb{N}$. Hence $j=0$, $k=1$ and so $\tau^a\beta\in G_{b,0}$. Given that $\tau^b\beta\in G_{b,0}$, we have $\tau^{a-b}\in\langle\sigma\rangle$ and so $a=b$.
	\end{proof}

	\begin{theorem}\label{ker lambda =p}
The skew braces of size $pq$ with additive group $M$ and $|\ker{\lambda}|=p$ are $A_\gamma=(A,+,\circ)$ for $1<\gamma\leq q$ where 
\begin{eqnarray}
  %
    \begin{pmatrix} n \\ m \end{pmatrix} + \begin{pmatrix} s \\ t \end{pmatrix} =\begin{pmatrix} n+g^{m}s \\ m+t \end{pmatrix},\qquad 
\begin{pmatrix} n \\ m \end{pmatrix} \circ \begin{pmatrix} s \\ t \end{pmatrix} =\begin{pmatrix} n+\left(g^{\gamma}\right)^m s \\ m+t \end{pmatrix} ,   \label{formula circ ker=p}
\end{eqnarray}
for every $0\leq n,s\leq p-1,\, 0\leq m,t\leq q-1$. In particular, they are bi-skew braces. 
 	\end{theorem}


\begin{proof} 
Let $G=G_{a,0}$ for $1\leq a\leq q-1$ be a regular subgroup of $\Hol(M)$ as in \eqref{reg groups ker=q}. According to Remark \ref{remark for lambdas}, $\lambda_{\sigma}=\pi_2(\sigma)=1$ and $\lambda_{\tau^a}=\pi_2(\tau^a\beta)=\beta$. As $a\circ b=a+\lambda_a(b)$, we have that 
$$\underbrace{\tau^a\circ \tau^a\circ \ldots \circ \tau^a}_n=\tau^{na},\qquad \underbrace{\sigma\circ \sigma\circ \ldots \circ \sigma}_n=\sigma^{n}$$ for every $n\in \mathbb{N}$. Since $\lambda:(A,\circ)\longrightarrow \aut{A,+}$ is a group homomomorphism, then $$\lambda_{\sigma^n\circ \tau^{ma}}=\lambda_{\sigma^n\tau^{ma}}=\lambda_{\sigma}^n\lambda_{\tau^a}^m=\beta^{m}$$
and therefore $\lambda_{\sigma^n\tau^m}=\beta^{\frac{m}{a}}$ for every $0\leq n\leq p-1$ and $0\leq m\leq q-1$.
Denoting the operation of $M$ additively and using the identification $\sigma\mapsto \begin{pmatrix} 1\\ 0\end{pmatrix}$ and $\tau\mapsto \begin{pmatrix} 0\\ 1\end{pmatrix}$ we have:
\[
\begin{pmatrix} n \\ m \end{pmatrix} \circ \begin{pmatrix} s \\ t \end{pmatrix} =\begin{pmatrix} n \\ m \end{pmatrix} +\beta^{\frac{m}{a}}\begin{pmatrix} s \\ t \end{pmatrix} =\begin{pmatrix} n+g^{m\frac{a+1}{a}}s \\ m+t \end{pmatrix}
\]
for every $0\leq n,s\leq p-1$ and $0\leq m,t\leq q-1$. Since $1\leq a\leq q-1$ then $1< \frac{a+1}{a}\leq q$. Both the additive and the multiplicative group of $A_\gamma$ are semidirect products of cyclic groups, so according to Corollary \ref{biskew on cyclic}, they are bi-skew braces.
\end{proof}

\begin{lemma}\label{lem:ker=1} \cite[Lemma 5.4]{Byott_pq}
The regular subgroups of $\Hol(M)$ such that the image under $\pi_2$ has size $pq$ are 
\begin{equation}\label{reg groups ker=1}
G_{c,d}=\langle \sigma^{a_0}\alpha, \sigma^c\tau^d\beta \rangle
\end{equation}
	where either $d\neq q-1$ or $c=0$, provided $d\neq 0$. In particular, $e^\prime(M,M,pq)=p(q-2)+1$.
	\end{lemma}



\begin{proposition}
	There exist $q-1$ skew braces with additive group $M$ and $\ker{\lambda}=0$.
	
\end{proposition}
	\begin{proof}
We show that the groups $G_{0,d}$ with $1\leq d\leq q-1$ as defined in \eqref{reg groups ker=1} are the representatives of the orbits of regular groups such that the image under $\pi_2$ has size $pq$.

First, we conjugate any $G_{0,d}$ by an arbitrary $\varphi_{i,j}$:

\begin{align*}
\varphi_{i,j}G_{0,d}\varphi_{i,j}^{-1} &= \langle \varphi_{i,j}(\sigma^{a_0})\alpha^i, \varphi_{i,j}(\tau^d)\varphi_{g,(1-g)j} \rangle \\
 &= \langle (\sigma^{a_0}\alpha)^i, (\sigma^j\tau)^d\varphi_{g,(1-g)j} \rangle \\
 &= \langle \sigma^{a_0}\alpha, \sigma^{j\frac{g^d-1}{g-1}}\tau^d\varphi_{g,(1-g)j} \rangle.
\end{align*}
Multiplying by $(\sigma^{a_0}\alpha)^{(g-1)j}$ the second generator, we get
\[
\langle \sigma^{a_0}\alpha, \sigma^{j\frac{g^{d+1}-1}{g-1}}\tau^d\beta \rangle=G_{j\frac{g^{d+1}-1}{g-1},d}.
\]
Now, since $\frac{g^{d+1}-1}{g-1}\equiv 0 \pmod p$ if and only if $d=q-1$, we have that the orbit of $G_{0,q-1}$ has only one element and that the orbit of $G_{0,d}$ with $d\neq q-1$ has $p$ elements as $j\frac{g^{d+1}-1}{g-1}$ runs from $0$ to $p$.

\end{proof}

\begin{theorem} 
The skew braces of size $pq$ with additive group $M$ and $\ker{\lambda}=0$ are $A_\mu=(A,+,\circ)$ for $1< \mu\leq q$ where
\begin{eqnarray}
\begin{pmatrix} n \\ m \end{pmatrix} + \begin{pmatrix} s \\ t \end{pmatrix} =\begin{pmatrix} n+g^{m}s \\ m+t \end{pmatrix},\qquad 
\begin{pmatrix} n \\ m \end{pmatrix} \circ \begin{pmatrix} s \\ t \end{pmatrix} =\begin{pmatrix} g^t n+\left(g^\mu\right)^{m}s \\ m+t \end{pmatrix}\label{circle for ker=0}
\end{eqnarray}
for every $0\leq n,s\leq p-1,\, 0\leq m,t\leq q-1$.
	\end{theorem}

\begin{proof}
Let $G_{0,d}$ be a regular subgroup of $\Hol(M)$ as in \eqref{reg groups ker=1}. By virtue of Remark \ref{remark for lambdas} we have $\lambda_{\pi_1(x)}=\pi_2(x)$
for every $x\in G_{0,d}$. In particular,
$$\lambda_{\sigma^{a_0}}=\alpha,\quad \lambda_{\tau^d}=\beta.$$ Since $a\circ b=a+\lambda_a(b)$, it is easy to check that 
$$\underbrace{\tau^d\circ \tau^d\circ \ldots \circ \tau^d}_n=\tau^{nd},\qquad \underbrace{\sigma^{a_0}\circ \sigma^{a_0}\circ \ldots \circ \sigma^{a_0}}_n=\sigma^{na_0}$$ for every $n\in \mathbb{N}$. Using that $\lambda:(A,\circ)\longrightarrow \aut{A,+}$ is a group homomorphism, then $\lambda_{\sigma}=\alpha^{g-1}$ and $\lambda_{\tau}=\beta^{d^{-1}}$. On the other hand,
\begin{align*}
\sigma^n\circ \tau^{m}&=\sigma^n \alpha^{n(g-1)}(\tau^m)\\
&=\sigma^{n+(g^{m}-1)n}\tau^{m}\\
&=\sigma^{g^{m}n}\tau^{m}
\end{align*}
and therefore
\begin{equation*}
     \lambda_{\sigma^n \tau^m}=\lambda_{\sigma^{ng^{-m}}\circ \tau^{m}}=\lambda_{\sigma^{ng^{-m}}}\lambda_{\tau^m}=\alpha^{(g-1)ng^{-m}}\beta^{md^{-1}}
\end{equation*}
for every $0\leq n\leq p-1,\, 0\leq m\leq q-1$. 
Denoting the operation of $M$ additively and using the identification $\sigma\mapsto \begin{pmatrix} 1\\ 0\end{pmatrix}$ and $\tau\mapsto \begin{pmatrix} 0\\ 1\end{pmatrix}$ we have 
\begin{eqnarray}
\begin{pmatrix} n \\ m \end{pmatrix} + \begin{pmatrix} s \\ t \end{pmatrix} =\begin{pmatrix} n+g^{m}s \\ m+t \end{pmatrix},\qquad 
\begin{pmatrix} n \\ m \end{pmatrix} \circ \begin{pmatrix} s \\ t \end{pmatrix} =\begin{pmatrix} g^t n+g^{m\frac{d+1}{d}}s \\ m+t \end{pmatrix}\end{eqnarray}
for every $0\leq n,s\leq p-1,\, 0\leq m,t\leq q-1$.
Since $1\leq d\leq q-1$ then $1< \mu=\frac{d+1}{d}\leq q$ and so we have formula
\eqref{circle for ker=0} for the $\circ$ operation.
\end{proof}

\subsection*{Acknowledgments} 
This work was partially supported by UBACyT 20020171000256BA and PICT 2016-2481. The authors thank Leandro Vendramin for drawing their attention to this problem and Kayvan Nejabati Zenouz for his comments and suggestions. The authors also want to thank the anonymous referee for their constructive comments and remarks.


\bibliographystyle{abbrv}
\bibliography{refs}

\end{document}